\documentclass[a4paper,11pt]{article}

\usepackage[margin=1in]{geometry}
\usepackage{setspace}
\usepackage{amsmath}
\usepackage{amssymb}
\usepackage{amsthm}
\usepackage{graphicx}
\usepackage{float}
\usepackage{caption}
\usepackage{listings}
\usepackage{tikz}
\usepackage{enumitem}
\usepackage{xcolor}
\usepackage{mathrsfs}

\newtheorem{theorem}{Theorem}[section]

\newtheorem{lemma}[theorem]{Lemma}
\newtheorem{definition}[theorem]{Definition}

\numberwithin{figure}{section}

\providecommand{\keywords}[1]{\textbf{Keywords:} #1}

\usepackage{hyperref}
\hypersetup{
    pdftitle={A Note on the Rainich Problem for SU(2) Gauge},
    pdfauthor={Hanwen Liu}
}

\begin{document}

\title{\texorpdfstring{\textbf{A Note on the Rainich Problem for SU(2) Gauge}}{A Note on the Rainich Problem for SU(2) Gauge}}
\author{Hanwen Liu}
\date{}

\maketitle

\begin{abstract}
We provide a resolution to the non-Abelian Rainich problem. By canonically identifying traceless symmetric $(0,2)$-tensors with Hermitian forms on the vector bundle of chiral 2-forms, we define the internal square roots of a stress-energy tensor. We then prove that the existence of a local $\operatorname{SU}(2)$ Yang-Mills field with prescribed stress-energy tensor $T$ is equivalent to a single differential condition on internal square roots of $T$.
\end{abstract}

\begin{center}
\end{center}

\onehalfspacing
\raggedbottom

\section{Introduction and Background}

There has long been profound interest in achieving a geometric unification of fundamental physical forces.
In the context of general relativity and gauge theory, the Rainich problem \cite{Rainich} establishes a natural goal, namely finding an electromagnetic field strength whose stress-energy tensor exactly matches a target spacetime Ricci tensor.
Because the $\operatorname{U}(1)$ structure group of classical electromagnetism imposes severe algebraic restrictions on the admissible Ricci tensors, this classical inverse problem can typically only be solved under exceedingly stringent algebraic conditions on the spacetime metric \cite{Rainich}.

To overcome this inherent structural limitation, Rainich's program inspired various non-Abelian generalizations over the decades. Instead of standard Einstein-Yang-Mills theory, earlier approaches often relied on alternative geometric formulations, such as finding double anti-self-dual solutions within Yang's parallel-displacement gravity \cite{Mielke1981}, or utilizing higher-dimensional Kaluza-Klein and tensor dominance models \cite{Mielke2017}. Within the standard 4-dimensional framework, recent efforts have attempted to dynamically block-diagonalize the stress-energy tensor using local orthogonal planes of symmetry and first-order perturbative methods \cite{Garat2019}. However, finding an exact result within standard general relativity has remained mathematically elusive.

In this note, we resolve this impasse by decoupling the algebraic constraints from the differential geometry. Rather than working directly within the standard metric-variation framework, we utilize the chiral formulation of 4-dimensional Lorentzian geometry. By taking internal square roots of a target stress-energy tensor, we absorb the degrees of freedom that traditionally obstruct the Abelian theory. We demonstrate that the complete Einstein-Yang-Mills constraints for an $\operatorname{SU}(2)$ gauge group can be packaged into a single exact differential equation. This translates the inverse problem of non-Abelian geometrization into a direct integrability condition on the background spacetime. Crucially, because this closed-form equation relies exclusively on the exterior derivatives and linear algebraic commutators of the target geometry, it provides a highly computable, algorithmic criterion for verifying whether a given macroscopic stress-energy tensor originates from a non-Abelian gauge field.

\section{The Main Result}

Throughout this article, for any $\mathfrak{sl}_2(\mathbb{C})$-valued differential $p$-form 
$$
\Omega=\sum_{i=1}^n\omega_i\otimes X_i,
$$
we define its conjugate $\bar{\Omega}$ to be
$$
\sum_{i=1}^n\bar{\omega}_i\otimes \phi(X_i),
$$
where $\bar{\omega}_i$ is the usual complex conjugate of the complex $p$-form $\omega_i$, and $\phi\in\operatorname{Aut}(\mathfrak{sl}_2(\mathbb{C}))$ is the involution of $\mathfrak{sl}_2(\mathbb{C})$ fixing its compact real form $\mathfrak{su}(2)$.

We first establish the algebraic correspondence between the macroscopic stress-energy tensor and the microscopic chiral geometry of the spacetime.

\begin{definition}
For any 3+1 spacetime $(M,g)$, we shall always denote by $\Lambda^+$ the bundle of self-dual 2-forms and by $\Lambda^-$ the bundle of anti-self-dual 2-forms on $(M,g)$. Moreover, we also canonically identify each traceless symmetric $(0,2)$-tensor ${T}$ on $(M,g)$ with a section $\hat{{T}}$ of $\Lambda^+\otimes\Lambda^-$ via the Clebsch-Gordan decomposition. For a traceless symmetric $(0,2)$-tensor ${T}$ on $(M,g)$, an $\mathfrak{sl}_2(\mathbb{C})$-valued self-dual local 2-form $\omega$ on $(M,g)$ satisfying $\hat{{T}}=-\operatorname{tr}(\omega\otimes\bar{\omega})$ is called an internal square root of ${T}$. 

Here and henceforth, the matrix trace $\operatorname{tr}(-)$ is taken in the Lie algebra $\mathfrak{sl}_2(\mathbb{C})$.
\end{definition}

It is readily seen that, for a traceless symmetric $(0,2)$-tensor ${T}$ on a 3+1 spacetime, the induced section $\hat{{T}}$ is an Hermitian form, whose eigenvalues are descirbed in the following lemma.

\begin{lemma}\label{Energy_condition}
For a traceless symmetric $(0,2)$-tensor ${T}$ satisfying the strict dominant energy condition on a 3+1 spacetime $(M,g)$, the induced Hermitian form $\hat{{T}}$ is positive definite, and hence ${T}$ always admits internal square roots.
\end{lemma}
\begin{proof}
Denote by $-\rho$ the eigenvalue of ${T}$ associated with its timelike eigenvector, and denote by $p_1,p_2,p_3$ the eigenvalues of ${T}$ associated with its spacelike eigenvectors, representing the principal pressures. Since $\operatorname{tr}_g(T)=0$, we have $\rho=p_1+p_2+p_3$. By linear algebra, the eigenvalues of $2\hat{{T}}$ are precisely $\rho-p_1,\rho-p_2,\rho-p_3$, which are all positive by the energy condition.
\end{proof}

We are now in the position to state and prove our main theorem.

\begin{theorem}\label{Main}
For a traceless symmetric $(0,2)$-tensor ${T}$ satisfying the strict dominant energy condition on a 3+1 spacetime $(M,g)$, the following statements are equivalent: 
\begin{enumerate}
    \item [1.] the Rainich condition $d(\operatorname{ad}^{-1}_\omega(d\omega))+\omega+\bar{\omega}=\operatorname{ad}^{-1}_\omega(d\omega)\wedge\operatorname{ad}^{-1}_\omega(d\omega)$ holds for some internal square root $\omega$ of ${T}$;
    \item [2.] locally there exists an $\operatorname{SU}(2)$ Yang-Mills field on $(M,g)$ of which energy-momentum tensor is precisely equal to ${T}$,
\end{enumerate}
where $\operatorname{ad}_\omega\colon\Omega^1(M;\mathfrak{sl}_2(\mathbb{C}))\rightarrow\Omega^3(M;\mathfrak{sl}_2(\mathbb{C}))$ is the adjoint action $\operatorname{ad}_\omega=[-,\omega]$ of $\omega$. 

Moreover, when an internal square root $\omega$ of ${T}$ satisfies the Rainich condition, the 2-form $\omega+\bar{\omega}$ is then an $\operatorname{SU}(2)$ Yang-Mills field on $(M,g)$ of which energy-momentum tensor equals ${T}$.
\end{theorem}
\begin{proof}
We first prove that statement (1) implies statement (2).

Let $\omega$ be an internal square root of $T$ satisfying the Rainich condition. 
Because $T$ satisfies the energy condition, the Hermitian form $\hat{T}$ is positive definite by Lemma~\ref{Energy_condition}, which implies that the components of $\omega$ span the bundle $\Lambda^+$. Notice that the center of the simple Lie algebra $\mathfrak{sl}_2(\mathbb{C})$ is trivial. We therefore conclude that the adjoint action $\operatorname{ad}_\omega \colon \Omega^1(M; \mathfrak{sl}_2(\mathbb{C})) \rightarrow \Omega^3(M; \mathfrak{sl}_2(\mathbb{C}))$ is an isomorphism.
We may therefore define a connection 1-form $A := -\operatorname{ad}^{-1}_\omega(d\omega)$.

By definition, this yields $d\omega + [A, \omega] = 0$, which is equivalent to $D_A\omega = 0$, where $D_A$ denotes the gauge covariant derivative of the connection $A$. 
Substituting $A$ into the Rainich condition immediately gives 
$dA + A \wedge A = \omega + \bar{\omega}$. Therefore, the curvature $F_A := dA + A \wedge A$ of the connection $A$ evaluates to $F_A = \omega + \bar{\omega}$. 

By the Bianchi identity, we have $D_A F_A = 0$, which yields $D_A(\omega + \bar{\omega}) = 0$. 
Since $D_A\omega = 0$, it follows that $D_A\bar{\omega} = 0$. 
Taking the complex conjugate of $D_A\omega = 0$ yields $D_{\bar{A}}\bar{\omega} = 0$. 
Rearranging then gives $[\omega,A - \bar{A}] = 0$. Since $\operatorname{ad}_\omega$ is an isomorphism, the condition $[\omega,A - \bar{A}] = 0$ implies that $A = \bar{A}$. 
Thus, the connection 1-form $A$ takes values in the compact real form $\mathfrak{su}(2)$.

We now verify the Yang-Mills equations for the $\operatorname{SU}(2)$-connection $A$. Evaluating the covariant derivative of the Hodge star of the curvature gives $D_A \star F_A = D_A \star(\omega + \bar{\omega})$, where $\star$ is the Hodge star operator on $(M,g)$. 
Since $\omega$ is self-dual and $\bar{\omega}$ is anti-self-dual, we have $\star\omega = i\omega$ and $\star\bar{\omega} = -i\bar{\omega}$, which yields 
$$ D_A \star F_A = D_A(i\omega - i\bar{\omega}) = iD_A\omega - iD_A\bar{\omega} = 0-0 = 0. $$
Thus, the 1-form $A$ is an $\operatorname{SU}(2)$ Yang-Mills potential. Under the canonical identification via the Clebsch-Gordan decomposition, the energy-momentum tensor of a gauge field corresponds to $-\operatorname{tr}(F_A^+ \otimes F_A^-)$. Since $F_A^+ = \omega$ and $F_A^- = \bar{\omega}$, this rearranges to $\operatorname{tr}(\omega \otimes \bar{\omega}) = -\hat{T}$. Therefore, the energy-momentum tensor of $A$ is precisely $T$.

We now proceed to prove that statement (2) implies statement (1).

Assume there exists a local $\operatorname{SU}(2)$ Yang-Mills field on $(M,g)$ with potential $A$ and Faraday tensor $F$, of which energy-momentum tensor is $T$.
Then, we have that $$\omega := \frac{1}{2}(F - i\star F)$$ is an $\mathfrak{sl}_2(\mathbb{C})$-valued self-dual 2-form. 
The energy-momentum tensor condition then yields that  $\operatorname{tr}(\omega \otimes \bar{\omega}) = -\hat{T}$, securing that $\omega$ is an internal square root of $T$. 

Because $A$ is a Yang-Mills potential 1-form, its curvature satisfies the Bianchi identity $D_A F = 0$ and the Yang-Mills equation $D_A \star F = 0$, which implies $d\omega + [A, \omega] = D_A \omega = 0$ by linearity. 
As established earlier, the adjoint action $\operatorname{ad}_\omega$ is invertible, we can solve for the connection $A = -\operatorname{ad}_\omega^{-1}(d\omega)$. Since $F$ is an $\mathfrak{su}(2)$-valued 2-form and hence is in particular real, we have $F \equiv dA + A \wedge A = \omega + \bar{\omega}$. 
Substituting $A = -\operatorname{ad}_\omega^{-1}(d\omega)$ into the curvature equation of $A$ provides the Rainich condition 
$$ d(\operatorname{ad}^{-1}_\omega(d\omega))+\omega+\bar{\omega}=\operatorname{ad}^{-1}_\omega(d\omega)\wedge\operatorname{ad}^{-1}_\omega(d\omega). $$

The proof is therefore completed.
\end{proof}

\section{Concluding Remarks}

Theorem~\ref{Main} illustrates a rigidity phenomenon inherent to non-Abelian gauge theories. In the classical $\operatorname{U}(1)$ Rainich problem, the algebraic constraints of gravity leave a continuous duality phase undetermined, which must be fixed subsequently by the differential Maxwell equations. In contrast, the $\operatorname{SU}(2)$ theory developed here is algebraically rigid: The requirement that the internal square root $\omega$ satisfies $\hat{T} = -\operatorname{tr}(\omega \otimes \bar{\omega})$ initially suggests a larger $\operatorname{U}(3)$ gauge freedom. However, the implicit requirement that the components of $\omega$ satisfy the Lie algebra relations of $\mathfrak{sl}_2(\mathbb{C})$ algebraically shatters this unitary symmetry, collapsing the freedom exactly down to the physical $\operatorname{SU}(2)$ gauge group.

Beyond its structural rigidity, the Rainich condition established here offers a computable criterion for physicists. Traditionally, determining whether a given symmetric bilinear form is a Yang-Mills stress–energy tensor requires solving a nonlinear system of PDEs for an unknown connection. Theorem~\ref{Main} circumvents this integration process. Once a target tensor $T$ is decomposed to extract an internal square root $\omega$, verifying its Yang-Mills compatibility is reduced to a forward calculation. The criterion relies on algorithmic evaluations such as differentiation $d\omega$ and matrix inversion $\operatorname{ad}_\omega^{-1}$. This translates a notoriously difficult inverse problem into a direct, computationally efficient verification process on the algebra of forms.


\section*{Acknowledgement}

The author is deeply grateful to Alexander Veselov and Evgeny Ferapontov for inspiring discussions. Special thanks should also go to Cheng He and Zongjian Han for careful proofreading.

\section*{Statements and Declarations}

No funding was received to assist with the preparation of this manuscript. The author certifies that the author has no affiliations with or involvement in any other organization or entity with any financial interest or non-financial interest in the subject matter or materials discussed in this manuscript. Data sharing is not applicable to this article as no datasets were generated or analysed during the current study.

\bibliographystyle{plain}
\bibliography{references}

\end{document}